\documentclass[12pt,a4paper,reqno]{amsart}
\usepackage[marginratio=1:1,totalwidth=15.75cm,totalheight=22.275cm]{geometry}
\usepackage[utf8]{inputenc}
\usepackage{amsmath,amssymb,amsthm}
\usepackage[alphabetic,nobysame]{amsrefs}
\usepackage[shortlabels]{enumitem}
\usepackage{graphicx}
\usepackage{csquotes}

\usepackage[dvipsnames]{xcolor}
\usepackage[colorlinks=true,linkcolor=blue!60!black,citecolor=teal!80!black,urlcolor=RoyalBlue]{hyperref}
\usepackage{tikz-cd}
\usepackage{amscd}

\numberwithin{equation}{section}

\usepackage{soul}

\def\Cx{\mathbb{C}}

\renewcommand{\hat}{\widehat}

\DeclareMathOperator{\inj}{inj}

\newcommand{\DD}{\mathbb{D}}

\DeclareMathOperator{\GO}{GO}

\theoremstyle{plain}
\newtheorem{thm}{Theorem}[section]
\newtheorem{lemma}[thm]{Lemma}

\newtheorem{prop}[thm]{Proposition}

\newtheorem{thmmain}{Theorem}

\theoremstyle{definition}
\newtheorem{defi}[thm]{Definition}

\theoremstyle{remark}

\title[Baker domains and disappearing orbits]{Baker domains and orbits disappearing to infinity}

\author[G. R. Ferreira]{Gustavo R. Ferreira}

\address{ Centre de Recerca Matemàtica\\ Bellaterra\\ Catalonia\\ Spain\\
	\textsc{\newline \indent 
		\href{https://orcid.org/0000-0002-7330-0018%
		}{\includegraphics[width=1em,height=1em]{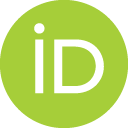} {\normalfont https://orcid.org/0000-0002-7330-0018}}
}}
\email{grodrigues@crm.cat}

\thanks{This project was supported by the European Union's Horizon Europe research and innovation programme under the Marie Skłodowska-Curie Grant Agreement No. 101208256}
\date{\today}

\begin{document}
\begin{abstract}
We study attracting orbits escaping to infinity in natural families of transcendental entire functions. We show that, if an attracting fixed point escapes to infinity while its multiplier tends to one, then the limiting function has a doubly parabolic Baker domain. Conversely, we show that any function with an invariant doubly parabolic Baker domain can be approximated locally uniformly by functions in its quasiconformal equivalence class having an attracting fixed point whose multiplier tends to one.

\end{abstract}
\maketitle

\section{Introduction}
In the theory of dynamical systems, a \textit{bifurcation} refers to a change in the qualitative behaviour of a system as its parameters are varied \cite{Kuz98}. The study of the various types of bifurcations is called \textit{bifurcation theory}, and it is often associated to understanding the system's periodic orbits and how they move through parameter space.

In holomorphic dynamics, a system takes the form of a complex-analytic function $f\colon S\to S$ being iterated, where $S$ is a Riemann surface. It was first observed by Fatou and Julia in the 1920s that $S$ can be divided into an open set where the iterates $\{f_n\}_n$ form a normal family in the sense of Montel, and a perfect set where the dynamics are chaotic in the sense of Devaney. The former is known as the \textit{Fatou set}, and denoted $F(f)$; the latter is called the \textit{Julia set}, $J(f)$. Both sets are completely invariant under $f$, meaning that if $U$ is a connected component of $f$ (henceforth, a \textit{Fatou component}) then $f(U)$ belongs to another Fatou component. If there exist $m > n \geq 0$ such that $f^m(U)\cap f^n(U)\neq\emptyset$, $U$ is said to be \textit{(pre-)periodic}; otherwise it is a \textit{wandering domain}. A \emph{periodic orbit} takes the form of a point $z\in S$ such that $f^n(z) = z$ for some $n\in\mathbb{N}$, and, when talking about parameter spaces, it is common to consider families of holomorphic maps $f_\lambda\colon S\to S$, where $\lambda$ moves in some complex manifold $M$, and $f_\lambda$ depends holomorphically on $\lambda$.

With this setup, many deep results in the area have been proved for rational maps by considering how periodic points move through both the phase and parameter spaces -- see, for instance, \cites{Shi87,McM00,Lev10}. The most celebrated and widely used application of this principle is probably Ma\~n\'e, Sad, and Sullivan's \cite{MSS83} use of analytic continuations of repelling periodic points and the $\lambda$-lemma to show density of structural stability for rational maps.

When put in the context of transcendental entire functions, one new phenomenon appears: periodic orbits \textit{disappearing} (or \textit{escaping}) \textit{to infinity}, i.e., a periodic orbit $\{z_1(t), z_2(t), \ldots, z_p(t)\}$ of the functions $f_t$ (all depending continuously on $t\in[0, 1]$) such that $f_t\to f$, a non-constant entire function, as $t\searrow 0$, but $z_i(t)\to\infty$ as $t\searrow 0$ for $1\leq i\leq p$. In families of \textit{finite-type} entire functions (i.e. those with only finitely many singular values), it was shown by Eremenko and Lyubich \cite{EL92}*{Theorem 2} that there are no orbits disappearing to infinity, allowing them to extend the results of Ma\~n\'e, Sad, and Sullivan to this setting. In the finite-type meromorphic context, Astorg, Benini, and Fagella \cite{ABF21} showed that, while periodic orbits \textit{can} escape the domain, this phenomenon is closely tied to a new kind of bifurcation involving \textit{virtual cycles}: asymptotic values whose orbit eventually lands on a pole. More precisely, they showed that if there is an attracting orbit $\{z_1(t), \ldots, z_p(t)\}$ such that $z_p(t)\to \infty$ as $t\searrow 0$ while the orbit's multiplier goes to zero, then one must have $z_1\to v$, an asymptotic value of $f$, and $z_{p-1}\to w$, a pole of $f$, creating the ``periodic orbit'' $v\mapsto \cdots \mapsto w\mapsto \infty \text{ ``$\mapsto v$''}$. They also showed that any map with a virtual cycle can be accessed within its parameter family by a path of maps with an attracting cycle tending to the virtual cycle while its multiplier tending to zero, showing that virtual cycles play a similar role to centres of hyperbolic components in the so-called ``shell components'' introduced in \cite{FK21}.

In this paper, we explore a new bifurcation involving periodic orbits disappearing to infinity in entire functions of infinite type. We show, in particular, that such orbits are closely tied to \textit{doubly parabolic Baker domains} (see Subsection \ref{ssec:BDs} for a definition): periodic Fatou components where orbits escape to the essential singularity at infinity in a similar way to orbits approaching a parabolic cycle.

Before we state our results, however, we must clarify what parameter space we are working in. The most ``general'' yet ``natural'' parameter space for an entire function $f$ is its \textit{quasiconformal equivalence class} $M_f$, the set of entire functions of the form $g = \psi\circ f\circ\varphi^{-1}$, where $\psi$ and $\varphi$ are quasiconformal homeomorphisms of $\Cx$. First studied by Eremenko and Lyubich \cite{EL92} in the context of finite-type maps, it has many properties that commend it as a general parameter space (see e.g. \cites{Rem09,ER15,ABF21}), and it admits the structure of a complex manifold (of possibly infinite dimension) for almost any entire function; see \cite{FvS25}*{Theorem A}. Holomorphic slices and submanifolds of $M_f$ are called \textit{natural families}. With this in mind, we can now state our results.

First, we show that any attracting fixed point escaping to infinity while its multiplier tends to one gives rise to a doubly parabolic Baker domain. This can be thought of as a ``Baker domain'' version of \cite{FK21}*{Theorem C}.
\begin{thmmain}\label{thm:escape}
Let $f$ be a transcendental entire function. Let $f_{\lambda(t)}$, $t\in[0, 1]$, be a path in $M_f$, with $f_{\lambda(t)}\to f$ as $t\searrow 0$.  Assume $f_{\lambda(t)}$ has a fixed point $z_t$ disappearing to infinity for $t = 0$. Assume that:
\begin{itemize}
    \item $z_t$ is attracting for $t\in (0, 1]$;
    \item There exists a singular value $v_t$ of $f_t$ in the basin of $z_t$ and such that the path $t\mapsto v_t$ is continuous for $t\in(0, 1]$;
    \item The multiplier $\rho_t$ of $z_t$ satisfies $\rho_t\nearrow 1$ horocyclically\footnote{See Definition \ref{def:horo}.} as $t\to 0$.
\end{itemize}
Then, $f$ has a doubly parabolic Baker domain.
\end{thmmain}

Our other result is a sort of converse to Theorem \ref{thm:escape}, in the spirit of \cite[Theorem~B]{ABF21}.
\begin{thmmain}\label{thm:perturb}
Let $f$ be a transcendental entire function with an invariant doubly parabolic Baker domain $U$. Then, there exists a path $f_{\lambda(t)}\in M_f$, $t\in[0, 1/2]$, such that:
\begin{enumerate}[label={\em (\arabic*)}]
    \item $f_{\lambda(0)} = f$;
    \item For $t\in (0, 1/2]$, $f_{\lambda(t)}$ has an attracting fixed point $z_t$ with multiplier $\rho_t$, and:
    \begin{enumerate}[label={\em (\alph*)}]
        \item $z_t$ disappears to infinity at $t = 0$;
        \item $|\rho_t|\nearrow 1$ as $t\to 0$.
    \end{enumerate}
\end{enumerate}
\end{thmmain}

We point out that, unlike in \cite{ABF21}*{Theorem~B}, our result does not say that $f$ can be approximated by such a path $f_{\lambda(t)}$ for \textit{any} natural family containing $f$; it takes advantage of the fact that we are operating in $M_f$. In fact, a perfect copy of \cite{ABF21}*{Theorem~B} is not possible, as the example $f_\lambda(z) = e^{-z} + z + \lambda$ shows: the Fatou set of $f_1$, the famous Fatou's function, is composed of a single doubly parabolic Baker domain, and the same is true for any $f_\lambda$ for $\lambda\approx 1$. The attracting fixed point escaping to infinity only becomes apparent in the ``transversal'' family $f_\alpha = \alpha\cdot f_1$, $\alpha\in\Cx^*$.

We also note that this link between orbits disappearing to infinity and Baker domains provides a stark contrast to the case of finite-type meromorphic functions, where attracting periodic orbits disappearing to infinity give rise to a virtual cycle. Unlike in our case, finite-type meromorphic functions cannot have Baker domains \cite{RS99}, and so the attracting basin disappears to infinity along with the periodic orbit.

Finally, we remark on the condition $\rho_t\nearrow 1$, which shows up in both Theorems \ref{thm:escape} and \ref{thm:perturb}. While it seems somewhat arbitrary, there are no examples of periodic orbits of entire functions escaping to infinity without satisfying it -- in fact, a quick calculation shows that it is sometimes necessary:
\begin{prop}\label{prop:uc}
Let $\{f_\lambda\}_{\lambda\in M}$, where $M$ is some complex manifold, be a natural family. Assume that there exists $M > 0$ such that
\[ \frac{\partial f_\lambda}{\partial\lambda}(z) \leq M \text{ for all $z\in\Cx$ and $\lambda\in M$.} \]
Then, if a periodic orbit disappears to infinity, its multiplier tends to one.
\end{prop}

The structure of this paper is as follows. First, in Section \ref{sec:prelim}, we introduce the required background (and notation) for hyperbolic geometry, grand orbits, horocyclic convergence, and Baker domains. Sections \ref{sec:escape} and \ref{sec:perturb} prove Theorems \ref{thm:escape} and \ref{thm:perturb}, respectively, and Section \ref{sec:uniform} has a proof of Proposition \ref{prop:uc}.

\section*{Acknowledgements}
I'd like to thank Núria Fagella for many insightful discussions.

\section{Preliminaries}\label{sec:prelim}
\subsection{Hyperbolic geometry}\label{ssec:hypgeom}
We define the \textit{hyperbolic density} in the unit disc to be
\[ \rho_\DD(z) = \frac{2}{1 - |z|^2} \text{ for }z\in\DD. \]
The resulting conformal metric $\rho_\DD|dz|$ is a complete Riemannian metric with constant curvature $-1$, called the \textit{hyperbolic metric} of $\DD$. This metric is invariant under M\"obius self-maps of the disc, and so by the uniformisation theorem any Riemann surface uniformised by the unit disc admits a complete Riemannian metric of curvature $-1$, called its \textit{hyperbolic metric}; see e.g. \cites{KL07,BM06} for more details. We will also use the notation $d_S$ for the hyperbolic distance on a hyperbolic surface $S$, as well as $\ell_S$ for the hyperbolic length of a curve in $S$.

Of interest to us are basic estimates on the hyperbolic density of plane domains -- starting with the following ``Harnack-type'' inequality due to Benini \textit{et al.}
\begin{lemma}[\cite{BEFRS22}*{Lemma 4.1}]\label{lem:harnack}
Let $\Omega\subsetneq\Cx$ be a simply connected domain. Then,
\[ e^{-2d_\Omega(z, w)} \leq \frac{\rho_\Omega(z)}{\rho_\Omega(w)} \leq e^{2d_\Omega(z, w)} \text{ for all $z, w\in\Omega$}. \]
\end{lemma}
Lemma \ref{lem:harnack} can be used to prove the following estimate, which features in the proofs of \cite{Fer22}*{Theorem 5.1} and \cite{BEFRS24}*{Theorem 3.3}.
\begin{lemma}\label{lem:delta}
Let $\Omega\subsetneq\Cx$ be a domain, and let $\tilde\Omega$ denote its topological convex hull (i.e. the union of $\Omega$ and all its bounded complementary components). Let $\delta_{\tilde\Omega}(z) := \mathrm{dist}(z, \partial\tilde\Omega)$ for all $z\in \Omega$. Then,
\[ |z - w| \leq 2d_\Omega(z, w)e^{2d_\Omega(z, w)}\delta_{\tilde\Omega}(z) \text{ for any $z, w\in\Omega$.} \]
\end{lemma}

We will also need the concept of injectivity radius:
\begin{defi}
Let $S$ be a hyperbolic surface, and let $p\in S$. The \textit{injectivity radius} of $S$ at $p$ is given by
\[ \inj_S(p) = \frac{1}{2}\min\{\ell_S(\gamma)\colon \text{$\gamma\subset S$ is a non-trivial closed loop passing through $p$}\}. \]
\end{defi}

\subsection{Grand orbits, quotient surfaces, and horocyclic convergence}
A more geometric approach to Fatou components is to look at the corresponding \textit{quotient surfaces}. We start with the \textit{grand orbit} of a point $z$, which (given a function $f$) is the set of all points whose orbit eventually intersects that of $z$:
\[ \GO(z) = \{w\in\Cx\colon \text{$f^m(z) = f^n(w)$ for some $m, n\in\mathbb{Z}$}\}; \]
the grand orbit of a set $S\subset \Cx$ is similarly defined as the union of the grand orbits of all its points. It is easy to see that $z\sim w\Leftrightarrow \GO(z) = \GO(w)$ defines an equivalence relation, called the \textit{grand orbit relation}.

With this in mind, let us understand the structure of $\GO(U)/f$ for a Fatou component $U$ of $f$. Following McMullen and Sullivan \cite{MS98}, we say that the grand orbit relation in $U$ is \textit{discrete} if $\GO(z)$ is a discrete (in $U$) set for all $z\in U$, and \textit{indiscrete} otherwise. The presence of normal forms in periodic Fatou components means that, if $U$ is periodic, either all points in $U$ have a discrete grand orbit, or none does (the situation is more complicated for wandering domains; see \cites{EFGP24,FFP25}).

Finally, let $U$ be a periodic Fatou component of $f$, for which the grand orbit relation is discrete. The quotient $S := \GO(U)/f$ carries a natural Riemann surface structure, which is determined by the corresponding normal form -- for instance, if $U$ is attracting, then $S$ is a complex torus, and if $U$ is parabolic, $S$ is a cylinder.

Now, in the attracting case, the precise geometry of the torus $S$ is determined by the multiplier $\rho$ of the corresponding attracting periodic orbit. If one considers $f$ as belonging to a parameter family $f_t$, $t\in[0, 1]$, with the corresponding multipliers $\rho_t$ satisfying $\rho_t\nearrow 1$ as $t\to 0$ (say), one sees the complex tori $S_t$ degenerating. The precise manner of this degeneration depends on how $\rho_t$ approaches one; we need the following definition:
\begin{defi}\label{def:horo}
The multipliers $\rho_t$ tend to one \textit{horocyclically} as $t\to 0$ if any horodisc $H\subset\DD$ anchored at one contains $\rho_t$ for all sufficiently small $t$. Alternatively, $\rho_t$ tends to one horocylically as $t\to 0$ if
\[ \Re\left({\frac{1}{1 - \rho_t}}\right)\to +\infty\text{ as $t\to 0$.} \]
\end{defi}

The importance of horocyclic convergence is the following fact (see \cite{McM00}*{p. 561}).
\begin{lemma}\label{lem:horo}
Let $U_t$ be the attracting basin of an attracting orbit with multiplier $\rho_t$ of the function $f_t$. If $\rho_t\nearrow 1$ horocyclically, then the quotient surfaces $\GO(U_t)/f_t$ converge geometrically to a cylinder.
\end{lemma}

\subsection{Baker domains}\label{ssec:BDs}
A Baker domain of the entire function $f$ is a Fatou component where the iterates tend locally uniformly to infinity. It was shown by Cowen \cite{Cow81} and Baker and Pommerenke \cite{BP79} that there are three types of Baker domains (hyperbolic, simply parabolic, and doubly parabolic), which can be distinguished by their respective normal forms. Here, however, we are more concerned with how to differentiate them by the behaviour of the hyperbolic metric (see \cites{Kon99,BFJK15}):
\begin{thm}\label{thm:BDhypclass}
Let $U$ be an invariant Baker domain of the transcendental entire function $f$. Then:
\begin{itemize}
    \item $U$ is doubly parabolic if and only if $\lim_{n\to\infty} d_U(f^{n+1}(z), f^n(z)) = 0$ for some (and hence all) $z\in U$.
    \item $U$ is simply parabolic if and only if $\inf_{z\in U}\lim_{n\to\infty} d_U(f^{n+1}(z), f^n(z)) = 0$, but $\lim_{n\to\infty} d_U(f^{n+1}(z), f^n(z)) > 0$ for all $z\in U$.
    \item $U$ is hyperbolic if and only if $\inf_{z\in U}\lim_{n\to\infty} d_U(f^{n+1}(z), f^n(z)) > 0$.
\end{itemize}
\end{thm}

Another way of classifying way that will be useful to us is to look at the grand orbit structure:
\begin{lemma}[\cite{FH06}*{Proposition 1}]
Let $U$ be an invariant Baker domain of $f$. Then:
\begin{itemize}
    \item $U$ is doubly parabolic if and only if $U/f\simeq \Cx/\mathbb{Z}$ (a cylinder).
    \item $U$ is simply parabolic if and only if $U/f\simeq \DD^*$.
    \item $U$ is hyperbolic if and only if $U/f \simeq \{z\in\Cx\colon 1 < |z| < r\}$ for some $r > 1$.
\end{itemize}
\end{lemma}

\section{Proof of Theorem \ref{thm:escape}}\label{sec:escape}
First, notice that the singular value $v_t$ of $f_t$ converge to a singular value $v\in\Cx$ of $f$, by the fact that the maps $f_t$ form a natural family. Second, we look at the quotient surfaces $S_t := \hat U_t/f_t$, where $\hat U_t = U_t\setminus \GO(S(f_t|_{U_t}))$. Since $U_t$ is an attracting basin, each $S_t$ is a torus with punctures (possibly infinitely many). Since $\rho_t\to 1$ horocyclically, the surfaces $S_t$ converge geometrically (by Lemma \ref{lem:horo}) to $S^*$, a cylinder with punctures (again, possibly infinitely many). This supports evidence, at the quotient level, for $f$ to have a doubly parabolic Baker domain. The challenge now lies in transferring this ``quotient-level'' knowledge back to the actual functions. We accomplish this in three parts. First, we show that there exists an $f$-invariant continuum $C\subset\Cx$ joining $v$ to $\infty$. Second, we show that there exists $c > 0$ such that $V := \{z\in\Cx\colon d(z, C) < c\}\subset U_t$ for all $t$. Finally, we show that the resulting Fatou component $U$ of $f$ that contains $C$ is, in fact, a doubly parabolic Baker domain.

The first part is straightforward. We let $\phi_t\colon U_t\to\Cx$ denote the (global) linearising coordinates of $z_t$, i.e. holomorphic functions satisfying $\phi\circ f_t = \rho_t\cdot\phi$. These functions depend continuously on $t$, and therefore so do $w_t = \phi_t(v_t)$ and the line segment $[w_t, 0]$. By the commuting diagram $\phi_t\circ f_t = \rho_t\cdot\phi$, the curve $\tilde C_t := \phi_t^{-1}([w_t, 0])$ is an $f_t$-invariant real-analytic curve joining $v_t$ to $z_t$. Now, recall that $f_t$ is a natural family, and let $f_t = \psi_t\circ f\circ \varphi_t^{-1}$, where $\psi_t$ and $\varphi_t$ are quasiconformal homeomorphisms depending continuously on $t$ and satisfying $\psi_0 = \varphi_0 = Id$. The curve $C_t = \varphi_t^{-1}(\tilde C_t)$ is mapped by $f$ onto $\psi_t^{-1}(\tilde C_t)$. Since the Hausdorff topology on closed subsets of the Riemann sphere is compact, it is clear that $\tilde C_t$ converges to a continuum $C$, and (since $\psi_t$ and $\varphi_t$ converge to the identity) so do $\varphi_t^{-1}(\tilde C_t)$ and $\psi_t^{-1}(\tilde C_t)$. This continuum is also, by construction, $f$-invariant, and joins $v$ to $\infty$.

The second part is less straightforward, and starts with the following lemma.
\begin{lemma}
For any $z\in C$, either:
\begin{itemize}
    \item There exists $r > 0$ such that $D(z, r)\subset U_t$ for all sufficiently small $t$; or
    \item $z$ is a fixed point of $f$.
\end{itemize}
\end{lemma}
\begin{proof}
We recall the surfaces $S_t := \hat U_t/f_t$. If we let $\pi_t$ denote the projection from $\hat U_t$ to $S_t$, we have that $d_{\hat U_t}(z, f_t(z)) = 2\inj_{\hat U_t}(\pi_t(z))$; in particular, this is true for $z\in \tilde C_t\cap \hat U_t$. Since $S_t\to S^*$ geometrically, we have that $\inj_{\hat U_t}(\pi_t(z))$ converges to a non-zero real number (say $\alpha$). Now, Lemma \ref{lem:delta} says that
\[ |f_t(z) - z| \leq 2d_{\hat U_t}(z, f_t(z))e^{2d_{\hat U_t}(z, f_t(z))}\delta_{U_t}(z), \]
where $\delta_{U_t}(z)$ denotes the distance from $z$ to $\partial U_t$. Since the left-hand side converges to $|f(z) - z|$, and $d_{\hat U_t}(z, f_t(z))$ is uniformly bounded, we have that $\delta_{U_t}(z)$ has a uniform (in $t$) lower bound whenever $z$ is not a fixed point of $f$. Our claim follows.
\end{proof}

Now, the possibility of fixed points of $f$ on $C$ poses a challenge; let us see what this would mean for $f$. If $z\in C$ satisfies $f(z) = z$, this is necessarily an isolated fixed point of $f$, meaning that it lies at the end of a subcontinuum $C'\subset C$ that does lie in the Fatou set of $U_t$ for small $t$. We claim that these points are in fact in the Fatou set of $f$; indeed, if $C'$ intersects $J(f)$, then there are repelling periodic points of $f$ arbitrarily closed to $C'$. By analytic continuation of the repelling periodic points, there exist periodic points of $f_t$ for small $t$ that are also close to $C'$, and so intersect $U_t$; this is a contradiction.

So, if $C'\subset F(f)$ and $C$ is invariant, it must belong to either an invariant Fatou component $V$ or a 2-cycle of Fatou components $\{V_1, V_2\}$ ``attached'' at $z$. However, the convergence of the quotient surfaces $S_t$ to a punctured cylinder tells us that $V$ (or the cycle $\{V_1, V_2\}$) must come from an attracting fixed point in $U_t$ whose multiplier tends to one (or $-1$). However, the only ``available'' fixed point is $z_t$, which is escaping to infinity as $t\searrow 0$. This is a contradiction, and so $C$ cannot contain any fixed points of $f$.

Finally, the third step starts with realising that $C$ belongs to a Fatou component $U$ of $f$; this can again be accomplished by arguing by contradiction, using the analytic continuation of the repelling periodic points of $f$. Next, we must reassure ourselves that $U$ has a discrete grand orbit relation, whence our theorem will follow: the quotient surface $U/f$ will be realised as the limit of the surfaces $U_t/f_t$, which (as discussed above) is a cylinder with possibly infinitely many punctures. The only Fatou components satisfying this are parabolic basins and doubly parabolic Baker domains, and $U$ cannot be parabolic since the fixed point $z_t$ escaped to infinity.

Now, to see that the grand orbit relation in $U$ is discrete. The only Fatou components of entire functions with indiscrete grand orbit relations are superattracting basins, Siegel discs, and some kinds of wandering domains (see \cite{EFGP24}). The invariance of the continuum $C\subset U$ precludes both wandering domains and Siegel discs, and $U$ cannot be a superattracting basin because the superattracting fixed point would have to lie on $C$, leading to a contradiction with the attracting dynamics of $\tilde C_t \subset U_t$.

\section{Proof of Theorem \ref{thm:perturb}}\label{sec:perturb}
To prove Theorem \ref{thm:perturb}, we must first take a closer look at the behaviour of $f$ inside its Baker domain. We have:
\begin{lemma}\label{lem:derivative}
Let $f$ be a transcendental entire function with a doubly parabolic invariant Baker domain $U$. Let $\gamma\colon[0, +\infty)\to U$ be such that $\gamma(t)\to\infty$ as $t\to\infty$ and $f(\gamma)\subset\gamma$. Then, $|f'(\gamma(t))|\to 1$ as $t\to\infty$.
\end{lemma}
\begin{proof}
First, we look at the hyperbolic distortion of $f$ along $\gamma$:
\[ \|Df(\gamma(t))\|_U^U = \frac{|f'(\gamma(t))|\rho_U(f(\gamma(t)))}{\rho_U(\gamma(t))}; \]
we will show that both the hyperbolic distortion on the left-hand side and the quotient $\rho_U(f(\gamma(t)))/\rho_U(\gamma(t))$ tend to one as $t\to\infty$.

For the former, let $\varphi\colon\DD\to U$ be a Riemann map conjugating $f\colon U\to U$ to the inner function $g\colon\DD\to\DD$. The curve $\gamma$ is conjugated to a $g$-invariant curve $\tilde\gamma$ landing non-tangentially at the Denjoy--Wolff point $p$ of $g$. Since $U$ is a doubly parabolic Baker domain, Cowen's classification of the dynamics of inner functions tells us that $|g'(p)| = 1$ in the sense of angular derivatives; thus, in the expression
\[ \|Dg(\tilde\gamma(t))\|_\DD^\DD = \frac{|g'(\tilde\gamma(t))|\rho_\DD(g(\tilde\gamma(t)))}{\rho_\DD(\tilde\gamma(t))}, \]
we know that at least one of the terms on the right-hand side tends to one as $t\to\infty$.

For the remaining quotient, we recall that, since $U$ is doubly parabolic and $\varphi$ is a hyperbolic isometry, $d_\DD(\tilde\gamma(t), g(\tilde\gamma(t)))\to 0$ as $t\to\infty$ by Theorem \ref{thm:BDhypclass}; using Lemma \ref{lem:harnack}, we conclude that $\rho_\DD(g(\gamma(t)))/\rho_\DD(\gamma(t))\to 1$ as $t\to\infty$. Therefore, $\|Dg(\tilde\gamma(t))\|_\DD^\DD\to 1$ as $t\to\infty$, and so by the Schwarz--Pick lemma the same holds for $\|Df(\gamma(t))\|_U^U$.

To conclude the proof, we use once again the fact that $U$ is doubly parabolic, whence $d_U(f(\gamma(t)), \gamma(t))\to 0$ as $t\to\infty$, and apply Lemma \ref{lem:harnack}.
\end{proof}

We continue the proof of Theorem \ref{thm:perturb} by defining the functions
\[ h_t(z) = \frac{\gamma(1/t)}{f(\gamma(1/t))}f(z), t\in(0, 1/2]; \]
since $f(\gamma)\subset\gamma$, and $\gamma$ can be chosen to avoid the origin, these are entire functions varying continuously in $M_f$. The point $z_t = \gamma(1/t)$ is clearly fixed, and escapes the domain as $t\searrow 0$ by definition. To see that $h_t\to f$ as $t\searrow 0$, we notice that, by Lemma \ref{lem:derivative} and L'Hopital's rule, the quotient $\gamma(s)/f(\gamma(s))$ tends to one in modulus as $s\to\infty$.

Next the multiplier of $z_t$ is
\[ h_t'(z_t) = \frac{\gamma(1/t)}{f(\gamma(1/t))}f'(\gamma(1/t)); \]
by Lemma \ref{lem:derivative} and the previous observation, this tends to one in modulus as $t\searrow 0$. We are left to show that the fixed point $z_t$ is attracting.

To that end, consider the real-valued function $g_t(s) = |h_t(\gamma(s))|$; the fixed point $z_t = \gamma(1/t)$ becomes a fixed point $s_t = 1/t$, $t\in (0, 1/2]$, of $g_t$. Assuming that $|g_t'(s_t)| > 1$, we will find another fixed point that also escapes to infinity and is attracting. Asymptotically, we have
\[ \text{$\frac{g_t(s)}{s} \to \frac{|z_t|}{|f(z_t)|} < 1$ as $s\to+\infty$,} \]
so that if $|g_t'(s_t)| > 1$, we must by continuity have another fixed point $x_t > s_t$ of $g_t$, which must be non-repelling. The corresponding point $w_t = \gamma(x_t)$ is not necessarily a fixed point of $h_t$, but it does satisfy $|w_t| = |h_t(w_t)|$, so we once again modify our function to 
\[ \tilde h_t(z) = \frac{w_t}{h_t(w_t)}h_t(z) = \frac{\gamma(x_t)}{f(\gamma(x_t))}f(z). \]
The same argument as before shows that $\tilde h_t(z)\to f$ as $t\searrow 0$, that $w_t$ is a fixed point of $\tilde h_t$, and that $w_t\to\infty$ as $t\searrow 0$. Furthermore, we have this time that $|\tilde h_t'(w_t)| \leq 1$.

If it turns out that $|\tilde h_t'(w_t)| = 1$, we can further modify the function by defining $f_t(z) = (1 - \epsilon_t)(\tilde h_t(z) - w_t) + w_t$, where $\epsilon_t\in (0, 1)$ is chosen so that $\epsilon_t\cdot w_t\to 0$ as $t\searrow 0$. This new function is entire, belongs to $M_f$, and still fixes $w_t$; the multiplier of $w_t$ has modulus $(1 - \epsilon_t)|\tilde h_t'(w_t)| < 1$, so that $w_t$ is an attracting fixed point of $f_t$. Furthermore, by our choice of $\epsilon_t$, we have $f_t\to f$ locally uniformly as $t\searrow 0$. This finishes the proof.

\section{Natural families with bounded partial derivatives}\label{sec:uniform}
We prove Proposition \ref{prop:uc}.
\begin{proof}[Proof of Proposition \ref{prop:uc}]
A $p$-periodic orbit of $f_\lambda$ is defined by the equation
\[ f_\lambda^p(z_\lambda) = z_\lambda; \]
differentiating with respect to $\lambda$ yields
\[ \frac{\partial f_\lambda^p}{\partial\lambda}(z_\lambda) + (f_\lambda^p)'(z_\lambda)\frac{\partial z_\lambda}{\partial \lambda} = \frac{\partial z_\lambda}{\partial\lambda}. \]
The term $(f_\lambda^p)'(z_\lambda)$ is, by definition, the multiplier $\rho(z_\lambda)$, so we can rearrange the expression above into
\begin{equation}\label{eq:multiplier}
\frac{\partial z_\lambda}{\partial\lambda} = \frac{1}{1 - \rho(z_\lambda)}\frac{\partial f_\lambda^p}{\partial\lambda}(z_\lambda).
\end{equation}
Now, if there exists some path $t\mapsto \lambda(t)$, $t\in[0, 1]$, along which $z_{\lambda(t)}$ disappears to infinity while $t\nearrow 1$, the left-hand side of (\ref{eq:multiplier}) tends to infinity, and therefore so does the right-hand side. However, the partial derivative of $f_\lambda$ with respect to $\lambda$ is uniformly bounded, so we must have $\rho(z_{\lambda(t)})\to 1$.
\end{proof}

\bibliographystyle{amsalpha}
\bibliography{ref}

@article{EL92,
    title={Dynamical properties of some classes of entire functions},
    author={A. Eremenko and M. Y. Lyubich},
    journal={Ann. Inst. Fourier},
    volume={42},
    pages={989--1020},
    year={1992}
}

@article{FK21,
    title={Stable components in the parameter plane of transcendental functions of finite type},
    author={N. Fagella and L. Keen},
    journal={J. Geom. Anal.},
    year={2021},
    volume={31},
    pages={4816--4855}
}

@article{Shi87,
    title={On the quasiconformal surgery of rational functions},
    author={M. Shishikura},
    journal={Ann. Sci. {\'Ec.} Norm. Sup.},
    year={1987},
    volume={20},
    pages={1--29}
}

@article{FvS25,
    title={Holomorphic motions, natural families of entire maps, and multiplier-like objects for wandering domains},
    author={G. R. Ferreira and S. van Strien},
    journal={Math. Ann.},
    volume={392},
    pages={701--732},
    year={2025}
}

@article{Rem09,
    title={Rigidity of escaping dynamics for transcendental entire functions},
    author={L. Rempe},
    journal={Acta Math.},
    year={2009},
    volume={203},
    pages={235--267}
}

@article{Lev10,
    title={Multiplier of periodic orbits in spaces of rational maps},
    author={G. Levin},
    journal={Ergod. Th. Dyn. Sys.},
    year={2010},
    volume={31},
    pages={197--243}
}

@misc{EFGP24,
    title={Grand orbit relations in wandering domains},
    author={V. Evdoridou and N. Fagella and L. Geyer and L. {Pardo-Sim\'on}},
    year={2024},
    note={Available at \href{https://arxiv.org/abs/2405.15667}{arXiv:2405.15667}}
}

@article{ER15,
    title={On invariance of order and the area property for finite-type entire functions},
    author={A. Epstein and L. Rempe-Gillen},
    journal={Ann. Acad. Sci. Fen. Math.},
    volume={40},
    year={2015},
    pages={573--599}
}

@article{BFJK15,
    title={Absorbing sets and {B}aker domains for holomorphic maps},
    author={K. Bara{\'n}ski and N. Fagella and X. Jarque and B. Karpi{\'n}ska},
    journal={J. London Math. Soc.},
    year={2015},
    pages={144--162},
    volume={92}
}

@article{Fer22,
    title={Multiply connected wandering domains of meromorphic functions: internal dynamics and connectivity},
    author={G. R. Ferreira},
    journal={J. London Math. Soc.},
    year={2022},
    volume={106},
    pages={1897--1919}
}

@article{BEFRS24,
    title={Boundary dynamics for holomorphic sequences, non-autonomous dynamical systems and wandering domains},
    author={A. M. Benini and V. Evdoridou and N. Fagella and P. J. Rippon and G. M. Stallard},
    journal={Adv. Math.},
    year={2024},
    volume={446},
    pages={109673}
}

@book{KL07,
    title={Hyperbolic Geometry from a Local Viewpoint},
    author={L. Keen and N. Lakic},
    year={2007},
    publisher={Cambridge University Press}
}

@article{FH06,
    title={Deformation of entire functions with {B}aker domains},
    author={N. Fagella and C. Henriksen},
    journal={Disc. Cont. Dyn. Sys.},
    year={2006},
    pages={379--394},
    volume={15}
}

@book{Kuz98,
    title={Elements of Applied Bifurcation Theory},
    author={Y. A. Kuznetsov},
    publisher={Springer},
    edition={2nd},
    year={1998}
}

@article{MS98,
	author = {C. T. McMullen and D. P. Sullivan},
	journal = {Adv. Math.},
	pages = {351--395},
	title = {Quasiconformal homeomorphisms and dynamics {III}: the {Teichm\"uller} space of a holomorphic dynamical system},
	volume = {135},
	year = {1998}
}

@misc{FFP25,
    title={Teichm{\"u}ller spaces and normal forms associated to wandering domains},
    author={N. Fagella and G. R. Ferreira and L. Pardo-Sim{\'o}n},
    year={2025},
    note={Available at \href{https://arxiv.org/abs/2512.01136}{arXiv:2512.01136}}
}

@article{BP79,
    title={On the iteration of analytic functions in a halfplane. {II}},
    author={I. N. Baker and Ch. Pommerenke},
    journal={J. London Math. Soc.},
    year={1979},
    volume={20},
    pages={255--258}
}

@article{Kon99,
    title={Conformal conjugacies in {B}aker domains},
    author={H. K{\"o}nig},
    year={1999},
    journal={J. London Math. Soc.},
    volume={59},
    pages={153--170}
}

@article{Cow81,
    title={Iteration and the solution of functional equations for functions analytic in the unit disk},
    author={C. C. Cowen},
    journal={Trans. Amer. Math. Soc.},
    year={1981},
    volume={265},
    pages={69--95}
}

@article{McM00,
    author={C. T. McMullen},
    title={Hausdorff dimension and conformal dynamics {II}: geometrically finite rational maps},
    journal={Comment. Math. Helv.},
    volume={75},
    year={2000},
    pages={535--593}
}

@article{RS99,
    title={Iteration of a class of hyperbolic meromorphic functions},
    author={P. J. Rippon and G. M. Stallard},
    journal={Proc. Amer. Math. Soc.},
    year={1999},
    volume={127},
    pages={3251--3258}
}

@article{BEFRS22,
    title={Classifying simply connected wandering domains},
    author={A. M. Benini and V. Evdoridou and N. Fagella and P. J. Rippon and G. M. Stallard},
    year={2022},
    journal={Math. Ann.},
    volume={383},
    pages={1127--1178}
}

@incollection{BM06,
    title={The hyperbolic metric and geometric function theory},
    author={A. F. Beardon and D. Minda},
    booktitle={Proceedings of the International Workshop on Quasiconformal Mappings and their Applications},
    year={2006},
    publisher={New Delhi Alpha Science International},
    address={New Delhi}
}

@article{MSS83,
    title={On the dynamics of rational maps},
    author={R. Ma{\~n\'e} and P. Sad and D. P. Sullivan},
    journal={Ann. Sci. {\'E}c. Norm. Sup.},
    year={1983},
    volume={16},
    pages={193--217}
}

@misc{ABF21,
    title={Bifurcation loci of families of finite type meromorphic maps},
    author={M. Astorg and A. M. Benini and N. Fagella},
    year={2021},
    note={Available on \href{https://arxiv.org/abs/2107.02663}{arXiv:2107.02663}}
}

\end{document}